\documentclass[a4paper, 11pt]{amsart}

\usepackage[T1]{fontenc}
\usepackage{amssymb,amscd}

\usepackage[all]{xy}

\usepackage[a4paper,hmargin=2.5cm,vmargin=2.5cm]{geometry}
\usepackage{todonotes}
\usepackage{mathrsfs}
\usepackage{mathtools}
\usepackage{graphicx}

\usepackage{xcolor}
\colorlet{NavyBlue}{blue!70!black!90}
\colorlet{PineGreen}{black!50!green}
\usepackage{tikz}

\usepackage{hyperref}
\hypersetup{colorlinks=true, linkcolor=blue, citecolor=black!50!green}

\usepackage{ytableau}
\usepackage{booktabs}
\usepackage{array}

\theoremstyle{plain}
\newtheorem{theorem}{Theorem}[subsection]
\newtheorem{lemma}[theorem]{Lemma}

\newtheorem{prop}[theorem]{Proposition}
\newtheorem{corollary}[theorem]{Corollary}

\theoremstyle{remark}

\newtheorem{remark}[theorem]{Remark}
\newtheorem{question}[theorem]{Question}

\theoremstyle{definition}
\newtheorem{example}[theorem]{Example}
\newtheorem{definition}[theorem]{Definition}
\newtheorem{construction}[theorem]{Construction}

\theoremstyle{plain}
\newtheorem{theoremintro}{Theorem}

\newtheorem{corollaryintro}[theoremintro]{Corollary}

\newcommand{\cC}{\mathscr{C}}

\newcommand{\cL}{\mathscr{L}}
\newcommand{\cP}{\mathscr{P}}

\newcommand{\RR}{\mathbb{R}}
\newcommand{\ZZ}{\mathbb{Z}}

\DeclareMathOperator{\GL}{GL}

\DeclareMathOperator{\rk}{rk}
\DeclareMathOperator{\wt}{wt}

\newcommand{\GZ}{\mathit{GZ}}          
\newcommand{\cN}{\mathscr{N}}          
\newcommand{\cR}{\mathscr{R}}          
\newcommand{\cE}{\mathscr{E}}          
\newcommand{\HS}{\mathscr{HS}}         
\newcommand{\Free}{\mathrm{Free}}      
\newcommand{\bnu}{\boldsymbol{\nu}}    
\newcommand{\Lft}{\mathsf{L}}          
\newcommand{\Rgt}{\mathsf{R}}          

\DeclareMathOperator{\SVT}{SVT}

\DeclareMathOperator{\ex}{ex}

\begin{document}


\title[Set-valued tableaux and cells of Gelfand--Zetlin polytopes]
      {Set-valued tableaux and cells of Gelfand--Zetlin polytopes}

\author{Evgeny Smirnov}
\address{HSE University, Usacheva 6, 119048 Moscow, Russia}
\address{Independent University of Moscow, Bolshoi Vlassievskii pereulok 11, 119002 Moscow, Russia}
\address{Guangdong Technion -- Israel Institute of Technology, Daxue road 241, Shantou, Guangdong, 515063 China}
\email{esmirnov@hse.ru}


\begin{abstract}
Two combinatorial rules are known for the Grassmannian Grothendieck polynomial
$G^{(\beta)}_\lambda$: a sum over set-valued tableaux of shape $\lambda$, due to Buch, and a sum
over the efficient cells of a cellular decomposition of the Gelfand--Zetlin polytope
$GZ(\lambda)$, due to E.~Presnova and the author. All coefficients in both sums equal $1$. We
construct an explicit bijection between the two indexing sets which matches the summands term by
term, carrying the number of excess entries of a tableau to the dimension of the corresponding
cell; in particular the two rules are equivalent, either being deducible from the other. The
efficiency condition on cells turns out to be the column-strictness of tableaux. We then transport Yu's square-root crystal operators to the
cells and find that they respect dimension --- along a double $i$-string the cells alternate
between two consecutive dimensions --- but not incidence: consecutive cells of such a string need
not share a point, already for $\lambda=(2,1,0)$.
\end{abstract}

\maketitle



\section{Introduction}
\label{sec:intro}

Let $\lambda$ be a partition with at most $n$ parts. The Grassmannian Grothendieck polynomial
$G^{(\beta)}_\lambda(x_1,\dots,x_n)$ represents the class of the structure sheaf of a Schubert
variety in the connective $K$-theory of a Grassmannian~\cite{Buch02}, the grading variable $\beta$
going back to Fomin and Kirillov~\cite{FominKirillov94,FominKirillov96}; setting $\beta=0$ recovers
the Schur polynomial $s_\lambda$. Two combinatorial rules for these polynomials are compared in
this paper.

The first is due to Buch~\cite{Buch02}. A \emph{set-valued tableau} of shape $\lambda$ is a filling
of the cells of $\lambda$ by non-empty subsets of $[n]$ such that every choice of one element per
cell yields a semistandard Young tableau; write $\SVT_n(\lambda)$ for the set of these, $\wt(T)$ for
the vector counting the occurrences of each letter, and $\ex(T)=|\wt(T)|-|\lambda|$ for the number
of excess entries. Then
\begin{equation}\label{eq:buch-intro}
G^{(\beta)}_\lambda=\sum_{T\in\SVT_n(\lambda)}\beta^{\ex(T)}x^{\wt(T)}.
\end{equation}

The second is geometric, and is the main result of the joint paper of Presnova and the
author~\cite{PresnovaSmirnov24}. The Gelfand--Zetlin polytope $\GZ(\lambda)$ was introduced
in~\cite{GelfandZetlin50} to index a basis of the irreducible $\GL(n)$-module with highest weight
$\lambda$; it is also the moment polytope of the toric degeneration of the flag variety due to
Gonciulea and Lakshmibai~\cite{GonciuleaLakshmibai96}, and Kogan and Miller~\cite{KoganMiller05}
identified the faces of $\GZ(\lambda)$ that correspond to the degenerations of Schubert varieties.
Building on this, Kiritchenko, the author and Timorin~\cite{KST} expressed key polynomials as sums
over the integer points of suitable unions of such faces. The paper~\cite{PresnovaSmirnov24}
lifts that description one level, from integer points to cells: it constructs a cellular
decomposition $\cC$ of $\GZ(\lambda)$ whose cells $C_P$ are indexed by \emph{enhanced
Gelfand--Zetlin patterns}, that is, integer points of $\GZ(\lambda)$ in which some entries are
encircled and some pairs of neighboring entries in consecutive rows are joined by edges, subject
to eight axioms. The dimension of $C_P$ is the number $\rk P$ of entries of $P$ that are not
encircled, so the $0$-cells are exactly the integer points of $\GZ(\lambda)$. Each pattern is
assigned a monomial $x^P$, which is $0$ precisely for the patterns called \emph{inefficient}, and
\begin{equation}\label{eq:ps-intro}
G^{(\beta)}_\lambda=\sum_{P\in\cP^+(\lambda)}x^P,
\end{equation}
the sum being over the efficient patterns. Thus \eqref{eq:ps-intro} refines the classical count of
integer points of $\GZ(\lambda)$ in the same way that \eqref{eq:buch-intro} refines the tableau
formula for $s_\lambda$.

Both formulas have a common refinement. Lascoux polynomials, introduced by
Lascoux~\cite{Lascoux04} by means of divided difference operators, form a basis of
$\ZZ[\beta][x_1,x_2,\dots]$ which simultaneously generalizes the Grassmannian Grothendieck
polynomials and, at $\beta=0$, the key polynomials, that is, the characters of Demazure
modules~\cite{Demazure74,Andersen85}. Restricting the sum in \eqref{eq:ps-intro} to the cells
lying in a union of faces attached to a permutation $w$ yields a formula for all Lascoux
polynomials~\cite{PresnovaSmirnov24}, generalizing the description of key polynomials in~\cite{KST};
on the tableau side, Yu~\cite{Yu21} refines \eqref{eq:buch-intro} by summing over the set-valued
tableaux whose right key is bounded by a fixed key, which simultaneously generalizes Buch's rule
and the tableau rule of Lascoux and Sch\"utzenberger~\cite{LascouxSchutzenberger90} for key
polynomials. Other combinatorial models are available --- Buciumas, Scrimshaw and
Weber~\cite{BSW20} describe Lascoux polynomials by means of colored five-vertex models --- and
these polynomials are the subject of considerable current activity; we mention the proof by
Shimozono and Yu~\cite{ShimozonoYu23} of the ``Grothendieck to Lascoux'' conjecture of Reiner and
Yong~\cite{ReinerYong21}, and the work of Pan and Yu~\cite{PanYu23} on top-degree components.
In a different, more geometric relation between Gelfand--Zetlin polytopes and Grothendieck
polynomials, Monin and the author~\cite{MoninSmirnov26} realize the K-theory of the flag variety
$\GL(n)/B$ as a ring of difference operators on $\GZ(\lambda)$ and represent the structure sheaves
of Schubert varieties, whose classes are the Grothendieck polynomials, by explicit alternating sums
of face operators, lifting to K-theory the description of Schubert classes in~\cite{KST}.

All coefficients on the right-hand sides of \eqref{eq:buch-intro} and \eqref{eq:ps-intro} are equal
to $1$. The two indexing sets $\SVT_n(\lambda)$ and $\cP^+(\lambda)$ are therefore equinumerous,
and more precisely the two sums carry the same multiset of monomials; but this was known only
because both sides compute $G^{(\beta)}_\lambda$, and no map between the two sets was available.
Constructing one was raised as a question at the end of the introduction
of~\cite{PresnovaSmirnov24}. The main result of this paper answers it.

\begin{theoremintro}[Theorem~\ref{thm:surj}]\label{thm:A}
For every partition $\lambda$ with at most $n$ parts there is an explicit bijection
\[
\widetilde\Phi\colon\ \SVT_n(\lambda)\ \longrightarrow\ \cP^+(\lambda)
\]
such that for every set-valued tableau $T$,
\[
x^{\widetilde\Phi(T)}=\beta^{\ex(T)}x^{\wt(T)}
\qquad\text{and}\qquad
\dim C_{\widetilde\Phi(T)}=\ex(T).
\]
\end{theoremintro}

Thus a set-valued tableau is a cell of $\cC$, its excess entries are the dimensions along which the
cell is free to move, and an ordinary semistandard tableau is a vertex.

The bijection factors through an intermediate object which is the real content of the construction.
Call a position of a Gelfand--Zetlin pattern \emph{free} if its entry is strictly larger than its
upper-left neighbor, and call a \emph{marked pattern} a pair $(A,U)$ consisting of a pattern $A$
and a subset $U$ of its free positions. Marked patterns form a set $\cN(\lambda)$ which is
transparently described: there are $2^{|\Free(A)|}$ of them over each pattern $A$, with no further
conditions. Theorem~\ref{thm:A} is then the composite of two bijections onto $\cN(\lambda)$.

The first, treated in Section~\ref{sec:bijection}, goes from tableaux. Given $T$, take the minima of
its cells; these form a semistandard tableau, hence a chain of partitions, hence a Gelfand--Zetlin
pattern. The remaining, non-minimal letters have nowhere to hide: a letter $k+1$ occurring
non-minimally in row $r$ must occupy one prescribed cell of that row (Lemma~\ref{lem:host}), so the
excess entries of $T$ are recorded by a set of positions and nothing else. This gives a bijection
$\Phi\colon\SVT_n(\lambda)\to\cN(\lambda)$ transporting $\ex$ and $\wt$ to the rank and the exponents
of $x^P$.

The second, treated in Section~\ref{sec:surjectivity}, goes from patterns, and is where the work
lies. Forgetting the edges of an efficient enhanced pattern and remembering which entries are not
encircled gives a map $\Psi\colon\cP^+(\lambda)\to\cN(\lambda)$; it is injective because the edges
of an efficient pattern can be reconstructed from the circles by a recipe $\cR$
of~\cite{PresnovaSmirnov24}. Surjectivity is the assertion that $\cR$, applied to an
\emph{arbitrary} marked pattern, always returns a legitimate enhanced pattern. Seven of the eight
axioms, and efficiency, follow from a single observation, the Anchor Lemma~\ref{lem:anchor}: an
entry equal to its upper-left neighbor is encircled. The eighth axiom is harder, and reduces to a
graph-theoretic statement; it is proved in Section~\ref{ss:descent} by means of a Descent Lemma
which propagates connectivity from one row of a pattern to the row above it. The tie-breaking rule of $\cR$ turns out
to be forced rather than conventional: neither of the two constant alternatives satisfies the axiom
(Remark~\ref{rem:tiebreak}).

The two halves of the argument meet in a statement that explains the somewhat technical notion of
efficiency. The free condition on a position is exactly the column-strictness of set-valued tableaux
read through the dictionary between patterns and chains; the enhanced patterns excluded by
efficiency are precisely those whose prescribed excess letter would collide with the cell below its
host (Corollary~\ref{cor:whyeff}).

Neither \eqref{eq:buch-intro} nor \eqref{eq:ps-intro} is used anywhere in the proof: the maps are
constructed and inverted explicitly. Two consequences follow.

\begin{corollaryintro}[Corollary~\ref{cor:equiv}]\label{cor:B}
Buch's formula \eqref{eq:buch-intro} and the formula \eqref{eq:ps-intro}
of~\cite{PresnovaSmirnov24} are equivalent; each may be deduced from the other, the bijection
$\widetilde\Phi$ matching their summands term by term.
\end{corollaryintro}

\begin{corollaryintro}[Corollary~\ref{cor:count}]\label{cor:C}
$\bigl|\SVT_n(\lambda)\bigr|=\sum_{A}2^{|\Free(A)|}$, the sum being over all Gelfand--Zetlin
patterns $A$ with top row $\lambda$; and the number of set-valued tableaux with $\ex(T)=d$ is
$\sum_{A}\binom{|\Free(A)|}{d}$.
\end{corollaryintro}

Finally, the introduction of~\cite{PresnovaSmirnov24} expresses the expectation that the crystal
operators which Yu~\cite{Yu21} defines on set-valued tableaux, and whose squares are the ordinary
crystal operators, should admit a description in terms of Gelfand--Zetlin polytopes.
Theorem~\ref{thm:A} transports them to the cells of $\cC$, and Section~\ref{sec:crystal} examines
what they become. The outcome is mixed. On the one hand, dimension behaves perfectly: along a
double $i$-string the cells alternate between two consecutive dimensions
(Theorem~\ref{thm:altern}), so that $f'_i$ never preserves the dimension of a cell although its
square always does. On the other hand, the geometric statement one would hope for is false. Two
consecutive cells of a double $i$-string need not be incident, and need not even be adjacent: the
smallest counterexample already occurs for $\lambda=(2,1,0)$ and $n=3$ (Theorem~\ref{thm:counter}).
The expectation
of~\cite{PresnovaSmirnov24}, then, holds at the level of dimensions and fails at the level of the
face poset, and the question of what should replace it is one of several collected in
Section~\ref{ss:questions}.

The paper is organized as follows. Section~\ref{sec:background} fixes conventions and recalls what
is needed from~\cite{Buch02} and~\cite{PresnovaSmirnov24}; the only novelty there is the dictionary
\eqref{eq:dict} between patterns and chains of partitions, which is used throughout.
Sections~\ref{sec:bijection} and~\ref{sec:surjectivity} construct $\widetilde\Phi$ and prove
Theorem~\ref{thm:A}. Section~\ref{sec:example} works out the case $\lambda=(2,1,0)$, $n=3$ in full,
listing all $27$ cells and their tableaux. Section~\ref{sec:crystal} treats the crystal operators.

\subsection*{Acknowledgements}
I am grateful to Tianyi Yu for useful discussions. This research was partially supported by the Basic
Research Program of HSE University (HSE-BR-2025-84). These results were obtained, and the text of
this paper prepared, in collaboration with Claude (Opus~4.8 and Fable~5); the results have been
human-verified, and the author bears full responsibility for their correctness.


\section{Conventions and background}
\label{sec:background}

Throughout the paper $n\ge 1$ is a fixed integer and
$\lambda=(\lambda_1\ge\lambda_2\ge\dots\ge\lambda_n\ge 0)$ is a partition with at most $n$ parts.
We identify $\lambda$ with its Young diagram, drawn in the English convention, and index its cells
by pairs $(r,c)$ with $1\le r\le n$ and $1\le c\le\lambda_r$. We write $|\lambda|=\sum\lambda_r$.

Section~\ref{ss:chains} recalls Gelfand--Zetlin polytopes and introduces the description of
Gelfand--Zetlin patterns by chains of partitions, which is the form in which they will be compared
with tableaux. Section~\ref{ss:svt} recalls set-valued tableaux and Buch's formula, and
Section~\ref{ss:enhanced} recalls the enhanced patterns and the cellular decomposition
of~\cite{PresnovaSmirnov24}. Nothing in this section is new; the only novelty is the dictionary
\eqref{eq:dict}, which is a matter of bookkeeping but will be used constantly.

\subsection{Gelfand--Zetlin patterns as chains of partitions}
\label{ss:chains}

Let $d=\frac{n(n-1)}{2}$ and consider $\RR^{d}$ with coordinates $y_{ij}$ indexed by pairs of
positive integers with $i+j\le n$. Consider the triangular tableau
\begin{equation}\label{eq:gz-tableau}
	\begin{matrix}
		\lambda_n && \lambda_{n-1}&& \lambda_{n-2}&&\dots &&\lambda_1\\
		& y_{11} && y_{12} && \dots && y_{1,n-1}\\
		&& y_{21}&& \dots && y_{2,n-2}\\
		&&& \ddots & \vdots &\reflectbox{$\ddots$}\\
		&&&&y_{n-1,1}
	\end{matrix}
\end{equation}
and impose, for every small triangle
$\begin{smallmatrix} a && b \\ & c \end{smallmatrix}$ occurring in it, the inequalities
$a\le c\le b$. It is convenient to set
\[
y_{0,j}:=\lambda_{n+1-j},\qquad 1\le j\le n,
\]
so that the topmost row of \eqref{eq:gz-tableau} is $(y_{0,1},\dots,y_{0,n})$ and the inequalities
read
\begin{equation}\label{eq:gz}
y_{i-1,j}\ \le\ y_{ij}\ \le\ y_{i-1,j+1},\qquad 1\le i\le n-1,\quad 1\le j\le n-i.
\end{equation}
Note that the top row is written in \emph{increasing} order.

\begin{definition}\label{def:gz}
The \emph{Gelfand--Zetlin polytope} $\GZ(\lambda)\subset\RR^{d}$ is the set of points satisfying
\eqref{eq:gz}. A \emph{Gelfand--Zetlin pattern} with top row $\lambda$ is an integer point of
$\GZ(\lambda)$; we denote its entries by $a_{ij}$, again with $a_{0,j}=\lambda_{n+1-j}$.
\end{definition}

We shall constantly refer to the two entries $a_{i-1,j}$ and $a_{i-1,j+1}$ lying immediately above
$a_{ij}$ as its \emph{upper-left} and \emph{upper-right} neighbors.

Gelfand--Zetlin patterns admit a second, equally classical description, which is the one that
matches tableaux. Let
\[
\bnu=\bigl(\varnothing=\nu^{(0)}\subseteq\nu^{(1)}\subseteq\dots\subseteq\nu^{(n)}=\lambda\bigr)
\]
be a chain of partitions in which every skew shape $\nu^{(k)}/\nu^{(k-1)}$ is a horizontal strip,
that is,
\begin{equation}\label{eq:hstrip}
\nu^{(k)}_r\ \ge\ \nu^{(k-1)}_r\ \ge\ \nu^{(k)}_{r+1}
\qquad\text{for all }k,r.
\end{equation}
Denote by $\HS(\lambda)$ the set of all such chains, the letters standing for
``horizontal strip''.

\begin{lemma}\label{lem:dict}
The assignment
\begin{equation}\label{eq:dict}
a_{ij}=\nu^{(k)}_r,\qquad\text{where } k=n-i \text{ and } r=k-j+1
\end{equation}
\textup{(}equivalently $i=n-k$, $j=k-r+1$\textup{)} is a bijection between $\HS(\lambda)$ and the
set of Gelfand--Zetlin patterns with top row $\lambda$. Under it, row $i$ of the tableau
\eqref{eq:gz-tableau}, read from right to left, is the partition $\nu^{(n-i)}$; and for every
$(i,j)$ with $i\ge 1$,
\begin{equation}\label{eq:upleft}
a_{i-1,j}=\nu^{(k+1)}_{r+1},\qquad a_{i-1,j+1}=\nu^{(k+1)}_{r},
\end{equation}
so that the upper-left neighbor of $\nu^{(k)}_r$ is $\nu^{(k+1)}_{r+1}$ and its upper-right
neighbor is $\nu^{(k+1)}_{r}$.
\end{lemma}

\begin{proof}
For $1\le i\le n-1$ and $1\le j\le n-i$ formula \eqref{eq:dict} gives $1\le k\le n-1$ and
$1\le r\le k$, and the two ranges correspond bijectively; the row index $k$ decreases as $i$
increases, and within a row the column index $r$ decreases as $j$ increases, which is the assertion
about reading rows from right to left. The identities \eqref{eq:upleft} are immediate: passing from
$(i,j)$ to $(i-1,j)$ replaces $k$ by $k+1$ and leaves $j$ unchanged, hence replaces $r=k-j+1$ by
$r+1$; passing to $(i-1,j+1)$ replaces $k$ by $k+1$ and $j$ by $j+1$, leaving $r$ unchanged.
Substituting \eqref{eq:upleft} into \eqref{eq:gz} turns the latter into \eqref{eq:hstrip}. Finally,
the boundary conventions match: $a_{0,j}=\lambda_{n+1-j}$ corresponds to $\nu^{(n)}=\lambda$, and
the entries $\nu^{(k)}_r$ with $r>k$, which do not occur in the tableau, vanish because
$\nu^{(k)}\subseteq\lambda$ has at most $k$ non-zero parts by \eqref{eq:hstrip}.
\end{proof}

We use the two descriptions interchangeably and refer to \eqref{eq:dict} as the \emph{dictionary}.
Under it, the condition that an entry be strictly larger than its upper-left neighbor reads
\begin{equation}\label{eq:free-both}
a_{ij}>a_{i-1,j}
\qquad\Longleftrightarrow\qquad
\nu^{(k)}_r>\nu^{(k+1)}_{r+1};
\end{equation}
this condition will play a central role in Section~\ref{sec:bijection}.

\subsection{Set-valued tableaux}
\label{ss:svt}

We follow~\cite{Yu21}. Write $[n]=\{1,\dots,n\}$.

\begin{definition}\label{def:svt}
A \emph{set-valued tableau} of shape $\lambda$ with entries in $[n]$ is a filling $T$ of the cells
of $\lambda$ by non-empty subsets of $[n]$ such that every choice of one element in each cell
produces a semistandard Young tableau. We write $\SVT_n(\lambda)$ for the set of these.
\end{definition}

Equivalently, $T$ is a set-valued tableau if and only if
\begin{equation}\label{eq:svt}
\max T(r,c)\ \le\ \min T(r,c+1)
\qquad\text{and}\qquad
\max T(r,c)\ <\ \min T(r+1,c)
\end{equation}
whenever both sides are defined. For $T\in\SVT_n(\lambda)$ let $\wt(T)=(\wt(T)_1,\dots,\wt(T)_n)$,
where $\wt(T)_m$ is the number of cells of $T$ containing $m$, and let
\[
\ex(T)=|\wt(T)|-|\lambda|=\sum_{(r,c)}\bigl(|T(r,c)|-1\bigr)
\]
be the number of \emph{excess} entries. Thus $\ex(T)=0$ precisely when $T$ is an ordinary
semistandard Young tableau.

The relevance of these objects is due to Buch.

\begin{theorem}[\cite{Buch02}]\label{thm:buch}
For every partition $\lambda$ with at most $n$ parts,
\begin{equation}\label{eq:buch}
G^{(\beta)}_\lambda(x_1,\dots,x_n)
=\sum_{T\in\SVT_n(\lambda)}\beta^{\ex(T)}x^{\wt(T)}.
\end{equation}
\end{theorem}

Yu~\cite{Yu21} extends \eqref{eq:buch} to all Lascoux polynomials, by summing over those
set-valued tableaux whose right key is bounded by a given key; he does so by equipping
$\SVT_n(\lambda)$ with a structure of an abstract Kashiwara crystal, whose raising and lowering
operators $e'_i$ and $f'_i$ satisfy $(e'_i)^2=e_i$ and $(f'_i)^2=f_i$. We shall not need the general
rule until Section~\ref{sec:crystal}, where the operators $f'_i$ are recalled in detail.

\subsection{Enhanced patterns, efficiency, and cells}
\label{ss:enhanced}

We now recall the construction of~\cite{PresnovaSmirnov24}. An \emph{enhancement} of a
Gelfand--Zetlin pattern consists of two kinds of data: some entries may be \emph{encircled}, and
some pairs of neighboring entries in consecutive rows may be joined by an \emph{edge}.

\begin{definition}[{\cite[Definition~4.1.1]{PresnovaSmirnov24}}]\label{def:enhanced}
A Gelfand--Zetlin pattern with top row $\lambda$, together with an enhancement, is an
\emph{enhanced Gelfand--Zetlin pattern} if the following conditions hold.
\begin{enumerate}
\item[(E1)]\label{cond:1} All entries of the top row are encircled.
\item[(E2)]\label{cond:2} Two entries joined by an edge are equal, and the lower one is encircled.
      \textup{(}The converse need not hold: equal neighboring entries need not be joined.\textup{)}
\item[(E3)]\label{cond:3} Two neighboring entries in a row are joined by edges to an entry above
      them if and only if they are joined by edges to an entry below them.
\item[(E4)]\label{cond:4} If two entries of the top row are equal, then the entry below them, which
      is equal to both, is encircled and joined to both by edges.
\item[(E5)]\label{cond:5} If $a<b$ and the pattern contains a triangle
      $\begin{smallmatrix} a&&b\\ &a\end{smallmatrix}$, then its bottom entry is encircled.
\item[(E6)]\label{cond:6} If $a<b$ and the pattern contains a triangle
      $\begin{smallmatrix} a&&b\\ &b\end{smallmatrix}$ whose bottom entry is encircled, then the
      two entries equal to $b$ are joined by an edge.
\item[(E7)]\label{cond:7} If a triangle $\begin{smallmatrix} a&&a\\ &a\end{smallmatrix}$ occurs and
      its two top entries can be joined by a path of edges, then its bottom entry is encircled and
      joined to both of them.
\item[(E8)]\label{cond:8} If a triangle $\begin{smallmatrix} a&&a\\ &a\end{smallmatrix}$ occurs and
      its bottom entry is encircled, then that entry is joined by an edge to at least one of the
      two top ones.
\end{enumerate}
The set of enhanced patterns with top row $\lambda$ is denoted $\cP(\lambda)$.
\end{definition}

In (E7) the path is understood to lie in the rows above the triangle; see the proof of
Lemma~\ref{lem:redundant} below, where this is used in the sharper form ``in rows $\le i-1$''.

\begin{definition}[{\cite[Definitions~4.1.5 and~4.1.6]{PresnovaSmirnov24}}]\label{def:rank}
The \emph{rank} $\rk P$ of an enhanced pattern $P$ is the number of its entries that are not
encircled. The pattern $P$ is \emph{inefficient} if it contains a triangle
$\begin{smallmatrix} a&&a\\ &a\end{smallmatrix}$ whose bottom entry is not joined by an edge to the
\emph{upper-right} one, and \emph{efficient} otherwise. We write $\cP^+(\lambda)\subseteq
\cP(\lambda)$ for the set of efficient enhanced patterns with top row $\lambda$.
\end{definition}

For an efficient pattern the edges carry no information beyond the circles.

\begin{lemma}[{\cite[Lemma~4.1.8]{PresnovaSmirnov24}}]\label{lem:redundant}
The edges of an efficient enhanced pattern are uniquely determined by the set of its encircled
entries, by the following procedure $\cR$. Declare equal neighboring entries of the top row to be
connected. Scan the rows from top to bottom; for an encircled entry $a_{ij}$ with $i\ge1$, compare
it with its upper neighbors $a_{i-1,j}$ and $a_{i-1,j+1}$:
\begin{itemize}
\item if it equals neither, draw no edge upwards from $a_{ij}$;
\item if it equals exactly one of them, join $a_{ij}$ to that one;
\item if it equals both, join $a_{ij}$ to both when $(i-1,j)$ and $(i-1,j+1)$ are already joined by
      a path of edges lying in rows $\le i-1$, and to the upper-right one alone otherwise.
\end{itemize}
Entries that are not encircled receive no edges.
\end{lemma}

We write $\cE_P(i,j)\subseteq\{\Lft,\Rgt\}$ for the set of edges going upwards from the position
$(i,j)$ of $P$, the symbols $\Lft$ and $\Rgt$ standing for the edges to $a_{i-1,j}$ and to
$a_{i-1,j+1}$ respectively; the subscript $P$ is omitted when no confusion can arise. In this
notation the last clause of Definition~\ref{def:rank} reads: $P$ is inefficient if and only if
$\Rgt\notin\cE(i,j)$ for some all-equal triangle with bottom entry $(i,j)$.

To an efficient enhanced pattern $P$ one associates the monomial
\begin{equation}\label{eq:xP}
x^P=\beta^{\rk P}x_1^{d_1}\cdots x_n^{d_n},
\qquad
d_{n+1-i}=S_{i-1}(P)-S_i(P)+D_i(P),
\end{equation}
where $S_i(P)$ is the sum of the entries in row $i$ of $P$, with $S_0(P)=|\lambda|$, and $D_i(P)$
is the number of entries in row $i$ that are not encircled. For an inefficient pattern one sets
$x^P=0$.

\begin{construction}[{\cite[Construction~4.2.1]{PresnovaSmirnov24}}]\label{constr:cells}
Let $P\in\cP(\lambda)$ have entries $a_{ij}$. To each coordinate $y_{ij}$ assign a relation:
\begin{enumerate}
\item if there is an edge from $a_{ij}$ up to $a_{i-1,j}$ \textup{(}resp.\ to
      $a_{i-1,j+1}$\textup{)}, impose $y_{ij}=y_{i-1,j}$ \textup{(}resp.\
      $y_{ij}=y_{i-1,j+1}$\textup{)};
\item if there are no edges going up from $a_{ij}$ and this entry is encircled, impose
      $y_{ij}=a_{ij}$;
\item if there are no edges going up from $a_{ij}$ and this entry is not encircled, impose the two
      strict inequalities
      \begin{enumerate}
      \item $a_{ij}-1<y_{ij}$ if $a_{ij}-a_{i-1,j}\ge2$, and $y_{i-1,j}<y_{ij}$ otherwise;
      \item $y_{ij}<y_{i-1,j+1}$ if $a_{i-1,j+1}=a_{ij}$, and $y_{ij}<a_{ij}$ otherwise.
      \end{enumerate}
\end{enumerate}
Let $\widehat{C_P}\subset\RR^d$ be the set defined by these relations, let $L$ be its affine span,
and put
\[
C_P=\widehat{C_P}\cap\bigl(\GZ(\lambda)\cap L\bigr)^{\circ},
\]
the relative interior being taken in $L$.
\end{construction}

\begin{theorem}[{\cite[Theorems~4.3.1 and~4.3.2]{PresnovaSmirnov24}}]\label{thm:ps}
The cells $C_P$, for $P\in\cP(\lambda)$, form a cellular decomposition $\cC$ of $\GZ(\lambda)$, and
$\dim C_P=\rk P$. Moreover
\begin{equation}\label{eq:ps}
G^{(\beta)}_\lambda(x_1,\dots,x_n)=\sum_{P\in\cP^+(\lambda)}x^P.
\end{equation}
\end{theorem}

The zero-dimensional cells of $\cC$ are exactly the integer points of $\GZ(\lambda)$, so
\eqref{eq:ps} refines the classical count of Gelfand and Zetlin~\cite{GelfandZetlin50} in the
same way that
\eqref{eq:buch} refines the tableau formula for a Schur polynomial: setting $\beta=0$ in either
kills all cells of positive dimension, respectively all tableaux with excess entries.

Comparing \eqref{eq:buch} with \eqref{eq:ps}, both sums have all coefficients equal to $1$, and both
are indexed by sets whose cardinality is therefore the same. The purpose of the next two sections is
to exhibit a bijection between these two indexing sets which matches the summands term by term.

\section{The bijection}
\label{sec:bijection}

By Lemma~\ref{lem:dict} a Gelfand--Zetlin pattern may be specified either by its entries $a_{ij}$
or by the corresponding chain $\bnu\in\HS(\lambda)$; we pass between the two without further
comment. Recall from \eqref{eq:free-both} that the inequality $a_{ij}>a_{i-1,j}$ between an entry
and its upper-left neighbor reads $\nu^{(k)}_r>\nu^{(k+1)}_{r+1}$ in the language of chains.

\subsection{From set-valued tableaux to marked Gelfand--Zetlin patterns}
\label{ss:phi}

\begin{definition}\label{def:free}
Let $A$ be a Gelfand--Zetlin pattern with top row $\lambda$ and entries $a_{ij}$. A position
$(i,j)$ with $i\ge1$ is \emph{free} for $A$ if
\[
a_{ij}>a_{i-1,j},
\]
that is, if the entry is strictly larger than its upper-left neighbor. We write $\Free(A)$ for the
set of free positions of $A$, and
\[
\cN(\lambda)=\bigl\{(A,U)\ :\ A \text{ a Gelfand--Zetlin pattern with top row } \lambda,\
U\subseteq\Free(A)\bigr\}
\]
for the set of \emph{marked} Gelfand--Zetlin patterns.
\end{definition}

The point of the definition is that, by Definition~\ref{def:rank} and Lemma~\ref{lem:redundant}, an
efficient enhanced pattern is determined by its underlying pattern together with the set of its
non-encircled entries; Proposition~\ref{prop:eff} below shows that this set is always contained in
$\Free$. Marked patterns are thus a candidate description of $\cP^+(\lambda)$ from which all
reference to edges has been removed.

We now attach a marked pattern to a set-valued tableau. Two pieces of data are extracted: the
minima of the cells, which produce the pattern, and the remaining entries, which produce the marks.

\begin{definition}\label{def:phi}
Let $T\in\SVT_n(\lambda)$.
\begin{enumerate}
\item The \emph{minimum tableau} $m(T)$ is the ordinary filling with $m(T)(r,c)=\min T(r,c)$.
\item For $0\le k\le n$ let $\nu^{(k)}(T)$ be the shape formed by those cells $(r,c)$ of $\lambda$
      with $m(T)(r,c)\le k$, and let $\bnu(T)=(\nu^{(0)}(T),\dots,\nu^{(n)}(T))$.
\item Call a pair $(k,r)$ with $1\le k\le n-1$ and $r\ge1$ \emph{excessive} for $T$ if the letter
      $k+1$ occurs as a non-minimal element of some cell in row $r$ of $T$. Let
      \[
      U(T)=\bigl\{(i,j)\ :\ i=n-k,\ j=k-r+1 \text{ for some pair } (k,r)
      \text{ excessive for } T\bigr\}
      \]
      be the corresponding set of positions, in the indexing of \eqref{eq:dict}.
\end{enumerate}
We set $\Phi(T)=\bigl(\bnu(T),U(T)\bigr)$.
\end{definition}

\begin{lemma}\label{lem:min}
For every $T\in\SVT_n(\lambda)$ the minimum tableau $m(T)$ is semistandard.
\end{lemma}

\begin{proof}
By \eqref{eq:svt} we have $\min T(r,c)\le\max T(r,c)\le\min T(r,c+1)$, so the rows of $m(T)$ weakly
increase, and $\min T(r,c)\le\max T(r,c)<\min T(r+1,c)$, so its columns strictly increase.
\end{proof}

\begin{lemma}\label{lem:chain}
For every $T\in\SVT_n(\lambda)$ we have $\bnu(T)\in\HS(\lambda)$.
\end{lemma}

\begin{proof}
Since the rows of $m(T)$ weakly increase, the cells of row $r$ with entry at most $k$ form an
initial segment of that row; let $\nu^{(k)}_r$ be its length, so that $\nu^{(k)}(T)$ is indeed a
Young diagram row by row. Clearly $\nu^{(k-1)}_r\le\nu^{(k)}_r$, and $\nu^{(0)}=\varnothing$,
$\nu^{(n)}=\lambda$ because all entries lie in $[n]$.

For the horizontal-strip condition \eqref{eq:hstrip}, suppose $m(T)(r+1,c)\le k$. As the columns of
$m(T)$ strictly increase, $m(T)(r,c)<m(T)(r+1,c)\le k$, whence $m(T)(r,c)\le k-1$. Therefore
$\nu^{(k)}_{r+1}\le\nu^{(k-1)}_r$, which is the right-hand inequality of \eqref{eq:hstrip}; taking
$k-1\le k$ also gives $\nu^{(k)}_{r+1}\le\nu^{(k)}_r$, so each $\nu^{(k)}$ is a partition.
\end{proof}

The following elementary observation is used repeatedly: an excess letter has no freedom about
where in its row it may sit.

\begin{lemma}[Host cell]\label{lem:host}
Fix a row $r$ of $T$ and an index $k\in[n-1]$. If the letter $k+1$ is a non-minimal element of the
cell $(r,c)$, then $c=\nu^{(k)}_r(T)$. In particular, for fixed $r$ and $k$ there is at most one
such cell, and it is the rightmost cell of row $r$ whose minimum is at most $k$.
\end{lemma}

\begin{proof}
As $k+1$ is non-minimal in $T(r,c)$ we have $\min T(r,c)<k+1$, i.e.\ $\min T(r,c)\le k$, and hence
$c\le\nu^{(k)}_r$. Suppose $c<\nu^{(k)}_r$. Then the cell $(r,c+1)$ exists and satisfies
$\min T(r,c+1)\le k$, whereas \eqref{eq:svt} gives $\min T(r,c+1)\ge\max T(r,c)\ge k+1$, a
contradiction.
\end{proof}

\begin{lemma}\label{lem:nec}
$U(T)\subseteq\Free\bigl(\bnu(T)\bigr)$ for every $T\in\SVT_n(\lambda)$.
\end{lemma}

\begin{proof}
Let $(k,r)$ be excessive and put $c=\nu^{(k)}_r$, so that by Lemma~\ref{lem:host} the letter $k+1$
is a non-minimal element of $T(r,c)$; in particular $c\ge1$. Assume for contradiction that
$\nu^{(k+1)}_{r+1}\ge c$. Then the cell $(r+1,c)$ exists and $\min T(r+1,c)\le k+1$, whereas
\eqref{eq:svt} gives $\min T(r+1,c)>\max T(r,c)\ge k+1$. This contradiction shows
$\nu^{(k+1)}_{r+1}<c=\nu^{(k)}_r$, which by \eqref{eq:free-both} is precisely the statement that the
corresponding position is free.
\end{proof}

\begin{theorem}\label{thm:bij1}
The map $\Phi\colon\SVT_n(\lambda)\to\cN(\lambda)$ is a bijection.
\end{theorem}

\begin{proof}
By Lemmas~\ref{lem:chain} and~\ref{lem:nec} the map indeed takes values in $\cN(\lambda)$.

\emph{Construction of the inverse.} Let $(\bnu,U)\in\cN(\lambda)$, and let $E$ be the set of pairs
$(k,r)$ corresponding to the positions of $U$ under \eqref{eq:dict}. For a cell $(r,c)$ of
$\lambda$ put
\[
\mu(r,c)=\min\bigl\{k\ :\ \nu^{(k)}_r\ge c\bigr\},
\]
which is well defined because $\nu^{(n)}=\lambda$, and define
\begin{equation}\label{eq:inverse}
T(r,c)=\{\mu(r,c)\}\ \cup\ \bigl\{k+1\ :\ (k,r)\in E,\ \nu^{(k)}_r=c\bigr\}.
\end{equation}

\emph{Step 1: $T\in\SVT_n(\lambda)$.} Every $T(r,c)$ is a non-empty subset of $[n]$: it contains
$\mu(r,c)$, and any adjoined element $k+1$ satisfies $k\le n-1$. Moreover
$\min T(r,c)=\mu(r,c)$, since an adjoined $k+1$ has $\nu^{(k)}_r=c$ and therefore
$\mu(r,c)\le k<k+1$.

For the row condition of \eqref{eq:svt}, let $c<\lambda_r$. Then $\mu(r,c)\le\mu(r,c+1)$ by
definition of $\mu$. If $k+1\in T(r,c)$ is adjoined, then $\nu^{(k)}_r=c<c+1$, so
$\mu(r,c+1)\ge k+1$. Hence $\max T(r,c)\le\mu(r,c+1)=\min T(r,c+1)$.

For the column condition, suppose the cell $(r+1,c)$ exists and put $k'=\mu(r+1,c)$, so that
$\nu^{(k')}_{r+1}\ge c$. By \eqref{eq:hstrip} we get $\nu^{(k'-1)}_r\ge\nu^{(k')}_{r+1}\ge c$ and
therefore $\mu(r,c)\le k'-1<k'$. If $k+1\in T(r,c)$ is adjoined, then $(k,r)\in E$, so the
corresponding position is free and $\nu^{(k+1)}_{r+1}<\nu^{(k)}_r=c$ by \eqref{eq:free-both}; hence
$\mu(r+1,c)\ge k+2>k+1$. Together these give $\max T(r,c)<\min T(r+1,c)$.

\emph{Step 2: $\Phi(T)=(\bnu,U)$.} Since $\min T(r,c)=\mu(r,c)$, we get
$\nu^{(k)}_r(T)=\#\{c:\mu(r,c)\le k\}=\nu^{(k)}_r$, so $\bnu(T)=\bnu$. The non-minimal elements
occurring in row $r$ of $T$ are exactly the adjoined letters $k+1$ with $(k,r)\in E$ — they are
non-minimal because $\mu(r,\nu^{(k)}_r)\le k$ — so $U(T)=U$.

\emph{Step 3: injectivity.} Let $T'\in\SVT_n(\lambda)$ satisfy $\Phi(T')=(\bnu,U)$. From
$\bnu(T')=\bnu$ we get $\min T'(r,c)=\mu(r,c)$ for every cell. By Lemma~\ref{lem:host} the
non-minimal elements of $T'$ in row $r$ are precisely the letters $k+1$ with $(k,r)\in E$, each
lying in the cell $(r,\nu^{(k)}_r)$. Comparing with \eqref{eq:inverse} gives $T'=T$.
\end{proof}

\begin{example}\label{ex:phi}
Let $n=3$, $\lambda=(2,1,0)$ and
\[
T=\ \ytableaushort{1{13},2}\ ,\qquad\text{that is } T(1,1)=\{1\},\ T(1,2)=\{1,3\},\ T(2,1)=\{2\}.
\]
Then $m(T)=\ytableaushort{11,2}$, so $\nu^{(1)}=(2,0)$, $\nu^{(2)}=\nu^{(3)}=(2,1)$ and the pattern is
\[
\begin{matrix}
0 && 1 && 2\\
 & 1 && 2\\
 && 2
\end{matrix}
\]
in the notation of \eqref{eq:gz-tableau}. The only non-minimal element of $T$ is the letter $3$ in
the cell $(1,2)$, giving the excessive pair $(k,r)=(2,1)$ and hence the single mark at position
$(i,j)=(1,2)$, the entry $a_{12}=2$. It is free, since its upper-left neighbor is $a_{0,2}=1$.
\end{example}

\subsection{Excess, weight, and the monomial of a pattern}
\label{ss:statistics}

We now check that $\Phi$ transports the two statistics carried by a set-valued tableau, namely
$\ex$ and $\wt$, to the two statistics carried by an enhanced pattern, namely its rank and the
exponents in \eqref{eq:xP}.

\begin{theorem}\label{thm:stats}
Let $T\in\SVT_n(\lambda)$ and $(\bnu,U)=\Phi(T)$. Then
\[
\ex(T)=|U|
\]
and, for $1\le m\le n$,
\[
\wt(T)_m=\bigl|\nu^{(m)}\bigr|-\bigl|\nu^{(m-1)}\bigr|
+\#\bigl\{r\ :\ (m-1,r)\text{ is excessive for } T\bigr\},
\]
the last term being empty for $m=1$.
\end{theorem}

\begin{proof}
By Lemma~\ref{lem:host} the non-minimal elements of $T$ are in bijection with the excessive pairs:
the pair $(k,r)$ corresponds to the letter $k+1$ lying in the cell $(r,\nu^{(k)}_r)$, and distinct
pairs give distinct occurrences. Hence
$\ex(T)=\sum_{(r,c)}\bigl(|T(r,c)|-1\bigr)=|U|$.

For the second identity, the letter $m$ occurs in $T$ either as the minimum of its cell or not. The
cells whose minimum equals $m$ are those counted by $|\nu^{(m)}|-|\nu^{(m-1)}|$, and the non-minimal
occurrences of $m$ correspond, again by Lemma~\ref{lem:host}, to the excessive pairs $(m-1,r)$, one
for each admissible $r$. Adding the two counts gives $\wt(T)_m$. For $m=1$ there are no non-minimal
occurrences, since $1$ is the minimum of any set containing it.
\end{proof}

\begin{corollary}\label{cor:monomial}
Let $P$ be an enhanced pattern whose underlying Gelfand--Zetlin pattern corresponds to $\bnu$ and
whose set of non-encircled entries is $U$, and suppose $(\bnu,U)\in\cN(\lambda)$. Then
\[
\rk P=\ex(T)
\qquad\text{and}\qquad
x^P=\beta^{\ex(T)}x^{\wt(T)},
\qquad\text{where } T=\Phi^{-1}(\bnu,U).
\]
\end{corollary}

\begin{proof}
The first equality is Definition~\ref{def:rank} together with Theorem~\ref{thm:stats}. For the
second, recall from Lemma~\ref{lem:dict} that row $i$ of $P$ consists of the parts of
$\nu^{(n-i)}$, whence $S_i(P)=\bigl|\nu^{(n-i)}\bigr|$; this is consistent at $i=0$, where
$S_0(P)=|\lambda|=\bigl|\nu^{(n)}\bigr|$. A position of $U$ lies in row $i$ exactly when the
corresponding excessive pair has $k=n-i$, so
\[
D_i(P)=\#\bigl\{r\ :\ (n-i,r)\text{ is excessive for } T\bigr\}.
\]
Substituting $m=n+1-i$ into the definition \eqref{eq:xP} of the exponents,
\[
d_m=S_{n-m}(P)-S_{n+1-m}(P)+D_{n+1-m}(P)
=\bigl|\nu^{(m)}\bigr|-\bigl|\nu^{(m-1)}\bigr|
+\#\bigl\{r:(m-1,r)\text{ excessive}\bigr\},
\]
which equals $\wt(T)_m$ by Theorem~\ref{thm:stats}.
\end{proof}

\subsection{Efficiency is column-strictness}
\label{ss:efficiency}

It remains to compare $\cN(\lambda)$ with $\cP^+(\lambda)$. One inclusion is a direct consequence of
the axioms.

\begin{prop}\label{prop:eff}
Let $P\in\cP^+(\lambda)$ have entries $a_{ij}$ and let $U$ be its set of non-encircled entries.
Then $U\subseteq\Free(P)$; that is, $a_{ij}>a_{i-1,j}$ for every $(i,j)\in U$.
\end{prop}

\begin{proof}
By \eqref{eq:gz} we always have $a_{ij}\ge a_{i-1,j}$, so suppose $(i,j)\in U$ and
$a_{ij}=a_{i-1,j}$; we derive a contradiction. By (E2) the lower endpoint of an edge is encircled,
so the entry $a_{ij}$, not being encircled, carries no edges at all. Two cases arise.

If $a_{i-1,j}<a_{i-1,j+1}$, then the triangle with top entries $a_{i-1,j}<a_{i-1,j+1}$ and bottom
entry $a_{ij}=a_{i-1,j}$ is of the shape appearing in (E5), which forces $a_{ij}$ to be encircled,
contradicting $(i,j)\in U$.

If $a_{i-1,j}=a_{i-1,j+1}$, then $a_{i-1,j}=a_{ij}=a_{i-1,j+1}$ is an all-equal triangle whose
bottom entry carries no edge, in particular no edge to the upper-right entry. By
Definition~\ref{def:rank} the pattern $P$ is then inefficient, again a contradiction.
\end{proof}

\begin{prop}\label{prop:inj}
The map
\[
\Psi\colon\cP^+(\lambda)\longrightarrow\cN(\lambda),
\qquad
P\longmapsto\bigl(\text{underlying pattern of } P,\ \text{non-encircled entries of } P\bigr)
\]
is well defined and injective.
\end{prop}

\begin{proof}
It is well defined by Proposition~\ref{prop:eff}. It is injective by Lemma~\ref{lem:redundant}: for
an efficient enhanced pattern the edges are reconstructed from the underlying pattern and the set of
encircled entries, so $P$ is recovered from $\Psi(P)$.
\end{proof}

That $\Psi$ is surjective is the substance of Section~\ref{sec:surjectivity}; granting it for the
moment, the composite $\Psi^{-1}\circ\Phi$ is the bijection announced in the introduction, and
Corollary~\ref{cor:monomial} says that it matches the summands of \eqref{eq:buch} and \eqref{eq:ps}
term by term.

We conclude this section by recording what Proposition~\ref{prop:eff} means on the tableau side. It
is not merely a technical restriction: it is the column-strictness of set-valued tableaux, seen
through the dictionary.

\begin{corollary}\label{cor:whyeff}
Let $(\bnu,U)\in\cN(\lambda)$ and $T=\Phi^{-1}(\bnu,U)$. A position $(i,j)$ of the pattern, with
$k=n-i$ and $r=k-j+1$, lies in $U$ if and only if the letter $k+1$ occurs as a non-minimal element
of the cell $(r,\nu^{(k)}_r)$ of $T$. The requirement $(i,j)\in\Free$, that is
$\nu^{(k)}_r>\nu^{(k+1)}_{r+1}$, holds if and only if inserting the letter $k+1$ into that cell does
not violate the column condition of \eqref{eq:svt}. Consequently the enhanced patterns excluded by
efficiency are exactly those whose prescribed excess letter would collide with the cell immediately
below its host.
\end{corollary}

\begin{proof}
The first assertion is Definition~\ref{def:phi} together with Lemma~\ref{lem:host}. For the second,
the cell immediately below the host cell $(r,\nu^{(k)}_r)$ exists and has minimum at most $k+1$
precisely when $\nu^{(k+1)}_{r+1}\ge\nu^{(k)}_r$, and in that case \eqref{eq:svt} fails for the pair
of cells $(r,\nu^{(k)}_r)$ and $(r+1,\nu^{(k)}_r)$ once $k+1$ is inserted above; this is the
computation carried out in the proofs of Lemma~\ref{lem:nec} and of Step~1 of
Theorem~\ref{thm:bij1}. The last assertion is Proposition~\ref{prop:eff} read backwards.
\end{proof}

\section{Surjectivity: the reconstruction recipe}
\label{sec:surjectivity}

By Proposition~\ref{prop:inj} the map $\Psi\colon\cP^+(\lambda)\to\cN(\lambda)$ is injective. This
section proves that it is surjective, and hence, with Theorem~\ref{thm:bij1}, that
$\widetilde\Phi=\Psi^{-1}\circ\Phi$ is a bijection from $\SVT_n(\lambda)$ onto $\cP^+(\lambda)$.

The task is concrete. Let $(A,U)\in\cN(\lambda)$ be a marked pattern, with entries $a_{ij}$ and with
$U\subseteq\Free(A)$ the set of positions to be left without a circle. Encircle every entry outside
$U$, in particular every entry of the top row, and run the procedure $\cR$ of
Lemma~\ref{lem:redundant} to produce a set of edges. We must show that the result is an enhanced
pattern in the sense of Definition~\ref{def:enhanced} and that it is efficient; its image under
$\Psi$ is then $(A,U)$ by construction.

It is worth emphasizing that $\cR$ is being used here in a way it was not designed for. In
Lemma~\ref{lem:redundant} it is applied to a pattern already known to be enhanced and efficient, and
merely recovers information that was there all along; here it is applied to an arbitrary marked
pattern, and the eight conditions must be verified from scratch. Seven of them turn out to be
automatic, for a single reason, which we isolate first.

\subsection{The anchor lemma and seven of the axioms}
\label{ss:anchor}

\begin{lemma}[Anchor lemma]\label{lem:anchor}
Let $(A,U)\in\cN(\lambda)$. If $a_{ij}=a_{i-1,j}$, then the entry $a_{ij}$ is encircled.
\end{lemma}

\begin{proof}
By Definition~\ref{def:free} a position of $U$ is free, that is, satisfies $a_{ij}>a_{i-1,j}$.
Hence $(i,j)\notin U$.
\end{proof}

The lemma is a triviality, but it is exactly the right triviality: every one of the conditions
(E4)--(E8) has a hypothesis forcing the bottom entry of a triangle to equal its upper-left
neighbor, and each of them asserts, among other things, that this entry is encircled. By
Corollary~\ref{cor:whyeff} the content of Lemma~\ref{lem:anchor} on the tableau side is the
column-strictness of set-valued tableaux; it is this that makes the axioms of
Definition~\ref{def:enhanced} self-enforcing.

\begin{theorem}\label{thm:reduce}
For every marked pattern $(A,U)\in\cN(\lambda)$, the enhancement produced by $\cR$ satisfies
conditions (E1), (E2), (E4), (E5), (E6), (E7), (E8) of Definition~\ref{def:enhanced}, and the
resulting pattern is efficient.
\end{theorem}

\begin{proof}
(E1) holds because $U$ contains no position with $i=0$, so all entries of the top row are encircled.

(E2) holds by construction: $\cR$ joins an entry only to an upper neighbor equal to it, and only
when the entry itself is encircled.

(E5). The hypothesis is that $a_{i-1,j}<a_{i-1,j+1}$ and $a_{ij}=a_{i-1,j}$. The latter equality is
that of Lemma~\ref{lem:anchor}, so $a_{ij}$ is encircled, which is the assertion.

(E4). Suppose $a_{0,j}=a_{0,j+1}$. By \eqref{eq:gz} we have
$a_{0,j}\le a_{1,j}\le a_{0,j+1}$, so $a_{1,j}=a_{0,j}=a_{0,j+1}$ and the triangle is all-equal. By
Lemma~\ref{lem:anchor} the entry $a_{1,j}$ is encircled, and its two upper neighbors are adjacent
equal entries of the top row, hence connected by the convention with which $\cR$ begins. The
all-equal branch of $\cR$ therefore draws both edges, as required.

(E6). Here $a_{i-1,j}<a_{i-1,j+1}=a_{ij}$ and the entry $a_{ij}$ is encircled. It equals its
upper-right neighbor and not its upper-left one, so $\cR$ draws exactly the edge to the
upper-right entry.

(E7). An all-equal triangle has $a_{ij}=a_{i-1,j}$, so $a_{ij}$ is encircled by
Lemma~\ref{lem:anchor}. If its two upper entries are joined by a path of edges in rows $\le i-1$,
the all-equal branch of $\cR$ draws both edges.

(E8). Same hypothesis; in either branch of the all-equal case $\cR$ draws at least the edge to the
upper-right entry.

Efficiency. Let an all-equal triangle with bottom entry $a_{ij}$ occur. As just observed, $a_{ij}$
is encircled and $\cR$ draws the edge to the upper-right entry, whichever branch applies. So the
configuration forbidden by Definition~\ref{def:rank} never arises.
\end{proof}

\subsection{The descent lemma and condition (E3)}
\label{ss:descent}

Only condition (E3) remains. It concerns a \emph{diamond}: four positions
\[
(i-1,j+1)\ \text{(top)},\qquad (i,j)\ \text{(left)},\qquad (i,j+1)\ \text{(right)},\qquad
(i+1,j)\ \text{(bottom)},
\]
with $1\le i\le n-2$ and $1\le j\le n-i-1$. In the notation introduced after
Lemma~\ref{lem:redundant}, the left and right entries are joined to the top one precisely when
$\Rgt\in\cE(i,j)$ and $\Lft\in\cE(i,j+1)$, and both are joined to the bottom one precisely when
$\cE(i+1,j)=\{\Lft,\Rgt\}$. Thus (E3) reads
\begin{equation}\label{eq:cond3}
\bigl[\,\Rgt\in\cE(i,j)\ \text{ and }\ \Lft\in\cE(i,j+1)\,\bigr]
\qquad\Longleftrightarrow\qquad
\cE(i+1,j)=\{\Lft,\Rgt\}.
\end{equation}

One implication is easy.

\begin{prop}\label{prop:cond3}
The implication $\Rightarrow$ in \eqref{eq:cond3} holds for the enhancement produced by $\cR$.
\end{prop}

\begin{proof}
The hypothesis gives $a_{ij}=a_{i-1,j+1}=a_{i,j+1}$, and \eqref{eq:gz} gives
$a_{ij}\le a_{i+1,j}\le a_{i,j+1}$, so all four entries of the diamond are equal. In particular
$a_{i+1,j}=a_{ij}$ is an all-equal triangle with bottom entry $(i+1,j)$, which is therefore
encircled by Lemma~\ref{lem:anchor}. Its two upper entries $(i,j)$ and $(i,j+1)$ are joined by the
path $(i,j)-(i-1,j+1)-(i,j+1)$, which lies in rows $\le i$; hence the all-equal branch of $\cR$
draws both edges from $(i+1,j)$.
\end{proof}

The converse is the only genuinely delicate point of the paper, and the reason is a mismatch of
scope. The all-equal branch of $\cR$ consults connectivity in rows $\le i-1$, whereas the hypothesis
$\cE(i+1,j)=\{\Lft,\Rgt\}$ supplies, as in the proof just given, a path in rows $\le i$. One must
show that the extra row buys nothing. This is the content of Lemma~\ref{lem:descent}, which rests in
turn on the following reading of the all-equal branch: a doubly-edged entry is a certificate of
connectivity one row higher.

\begin{lemma}\label{lem:both}
Suppose $\cE(i,k)=\{\Lft,\Rgt\}$. Then $a_{i-1,k}=a_{ik}=a_{i-1,k+1}$, and the positions $(i-1,k)$
and $(i-1,k+1)$ are joined by a path of edges lying in rows $\le i-1$.
\end{lemma}

\begin{proof}
Edges join equal entries, whence the displayed equalities. Two upward edges are drawn by $\cR$ only
in the all-equal branch: if $a_{ik}$ equalled exactly one of its upper neighbors only one edge
would be drawn, and if it equalled neither, none. The defining condition of that branch is precisely
the asserted connectivity. For $i=1$ the condition holds automatically, since $a_{0,k}=a_{0,k+1}$
are adjacent equal entries of the top row.
\end{proof}

\begin{lemma}[Descent lemma]\label{lem:descent}
Let $i\ge1$ and let $u,w$ be positions in rows $\le i-1$. If $u$ and $w$ are joined by a path of
edges lying in rows $\le i$, then they are joined by a path of edges lying in rows $\le i-1$.
\end{lemma}

\begin{proof}
Let $Q$ be such a path. Edges join positions in consecutive rows only, so no two positions of row
$i$ are adjacent; since the endpoints $u$ and $w$ lie in rows $\le i-1$, every position of $Q$ in
row $i$ is an interior vertex of $Q$. Such a position $(i,k)$ therefore meets two distinct edges of
$Q$, and its only available neighbors within rows $\le i$ are $(i-1,k)$ and $(i-1,k+1)$. Hence $Q$
traverses the detour $(i-1,k)-(i,k)-(i-1,k+1)$ and $\cE(i,k)=\{\Lft,\Rgt\}$. By
Lemma~\ref{lem:both} there is a path from $(i-1,k)$ to $(i-1,k+1)$ lying in rows $\le i-1$;
substitute it for the detour. Performing this replacement at every position of $Q$ in row $i$
yields a walk from $u$ to $w$ in rows $\le i-1$, and hence a path.
\end{proof}

\begin{theorem}\label{thm:cond3}
The implication $\Leftarrow$ in \eqref{eq:cond3} holds for the enhancement produced by $\cR$.
Consequently $\cR$ satisfies (E3), and, with Theorem~\ref{thm:reduce}, produces an efficient
enhanced pattern for every marked pattern $(A,U)\in\cN(\lambda)$.
\end{theorem}

\begin{proof}
Assume $\cE(i+1,j)=\{\Lft,\Rgt\}$ and set $v=a_{i+1,j}$.

\emph{Step 1: all four entries of the diamond equal $v$.} By Lemma~\ref{lem:both} we have
$a_{ij}=v=a_{i,j+1}$, and the positions $(i,j)$ and $(i,j+1)$ are joined by a path $Q$ of edges
lying in rows $\le i$. The inequalities \eqref{eq:gz} give $a_{ij}\le a_{i-1,j+1}\le a_{i,j+1}$,
whence $a_{i-1,j+1}=v$ as well. Thus $(i,j)$ equals its upper-\emph{right} neighbor and $(i,j+1)$
equals its upper-\emph{left} neighbor.

\emph{Step 2: $\Rgt\in\cE(i,j)$.} We first check that $(i,j)$ is encircled. An entry that is not
encircled receives no upward edges, and edges descending from row $i+1$ are unavailable inside rows
$\le i$; a non-encircled $(i,j)$ would therefore be an isolated vertex of the graph induced on rows
$\le i$, and could not be joined to $(i,j+1)$ by $Q$. Now $a_{ij}=v=a_{i-1,j+1}$. If
$a_{i-1,j}<v$, then $a_{ij}$ equals exactly one of its upper neighbors and $\cR$ draws the edge to
the upper-right one. If $a_{i-1,j}=v$, the triangle is all-equal and $\cR$ draws the edge to the
upper-right entry in either branch. In both cases $\Rgt\in\cE(i,j)$.

\emph{Step 3: $\Lft\in\cE(i,j+1)$.} The entry $a_{i,j+1}=v=a_{i-1,j+1}$ equals its upper-left
neighbor, so it is encircled by Lemma~\ref{lem:anchor}. If $a_{i-1,j+2}>v$, then it equals exactly
one of its upper neighbors and $\cR$ draws the edge to the upper-left one, as required. So assume
$a_{i-1,j+2}=v$. Then the triangle above $(i,j+1)$ is all-equal, and it suffices to prove that
$(i-1,j+1)$ and $(i-1,j+2)$ are joined by a path of edges in rows $\le i-1$, for then the all-equal
branch of $\cR$ draws both edges at $(i,j+1)$.

Write $Q=\bigl((i,j)=q_0,q_1,\dots,q_m=(i,j+1)\bigr)$. As no two positions of row $i$ are adjacent,
$m\ge2$, and
\[
q_1\in\bigl\{(i-1,j),\,(i-1,j+1)\bigr\},
\qquad
q_{m-1}\in\bigl\{(i-1,j+1),\,(i-1,j+2)\bigr\}.
\]
If $q_{m-1}=(i-1,j+1)$, then the last edge of $Q$ is the upward-left edge at $(i,j+1)$ and we are
done. So suppose $q_{m-1}=(i-1,j+2)$.

The subpath $q_1,\dots,q_{m-1}$ lies in rows $\le i$ and has both endpoints in row $i-1$, so by
Lemma~\ref{lem:descent} the positions $q_1$ and $q_{m-1}=(i-1,j+2)$ are joined by a path in rows
$\le i-1$. It therefore suffices to join $(i-1,j+1)$ to $q_1$ within rows $\le i-1$. If
$q_1=(i-1,j+1)$ this is vacuous. If $q_1=(i-1,j)$, then the first edge of $Q$ is the upward-left
edge at $(i,j)$, which together with Step 2 gives $\cE(i,j)=\{\Lft,\Rgt\}$; Lemma~\ref{lem:both}
then joins $(i-1,j)=q_1$ to $(i-1,j+1)$ within rows $\le i-1$, as needed.
\end{proof}

\begin{remark}\label{rem:tiebreak}
The proof explains why the tie-breaking rule of $\cR$ has to be exactly what it is, and neither of
the two constant alternatives will do. Lemma~\ref{lem:descent} is available only because a position
of row $i$ used as an interior vertex of a path must itself carry two edges, and therefore, by
Lemma~\ref{lem:both}, already certifies connectivity one row higher; this is what the connectivity
test in the all-equal branch provides. If that branch is replaced by ``always draw the edge to the
upper-right entry alone'', or by ``always draw both edges'', then \eqref{eq:cond3} fails.
\end{remark}

\subsection{Consequences}
\label{ss:consequences}

\begin{theorem}\label{thm:surj}
The map $\Psi\colon\cP^+(\lambda)\to\cN(\lambda)$ is a bijection. Consequently
\[
\widetilde\Phi:=\Psi^{-1}\circ\Phi\ \colon\ \SVT_n(\lambda)\ \longrightarrow\ \cP^+(\lambda)
\]
is a bijection satisfying
\[
x^{\widetilde\Phi(T)}=\beta^{\ex(T)}x^{\wt(T)}
\qquad\text{and}\qquad
\dim C_{\widetilde\Phi(T)}=\rk\widetilde\Phi(T)=\ex(T).
\]
\end{theorem}

\begin{proof}
Let $(A,U)\in\cN(\lambda)$ and let $P$ be the pattern $A$ with the entries outside $U$ encircled and
the edges drawn by $\cR$. By Theorems~\ref{thm:reduce} and~\ref{thm:cond3} we have
$P\in\cP^+(\lambda)$, and $\Psi(P)=(A,U)$ by construction; so $\Psi$ is surjective, and it is
injective by Proposition~\ref{prop:inj}. The map $\Phi$ is a bijection onto $\cN(\lambda)$ by
Theorem~\ref{thm:bij1}, so $\widetilde\Phi$ is a bijection. The identity of monomials is
Corollary~\ref{cor:monomial}, and the statement about dimensions combines it with
Theorem~\ref{thm:ps}.
\end{proof}

Nothing in the argument used either of the two enumerative formulas \eqref{eq:buch} and
\eqref{eq:ps}: the proof is a bijection between the two indexing sets, constructed and inverted
explicitly. The formulas may therefore be deduced from one another.

\begin{corollary}\label{cor:equiv}
Buch's Theorem~\ref{thm:buch} and the formula \eqref{eq:ps} of~\cite{PresnovaSmirnov24} are
equivalent: the bijection $\widetilde\Phi$ matches their summands term by term, so either may be
deduced from the other.
\end{corollary}

\begin{proof}
By Theorem~\ref{thm:surj} the map $\widetilde\Phi$ is a bijection from $\SVT_n(\lambda)$ onto
$\cP^+(\lambda)$ carrying $\beta^{\ex(T)}x^{\wt(T)}$ to $x^{\widetilde\Phi(T)}$. Hence
\[
\sum_{T\in\SVT_n(\lambda)}\beta^{\ex(T)}x^{\wt(T)}
=\sum_{P\in\cP^+(\lambda)}x^P
\]
as formal expressions, independently of the common value of the two sides.
\end{proof}

\begin{corollary}\label{cor:count}
For every partition $\lambda$ with at most $n$ parts,
\[
\bigl|\SVT_n(\lambda)\bigr|=\bigl|\cP^+(\lambda)\bigr|
=\sum_{A}2^{|\Free(A)|},
\]
the sum being over all Gelfand--Zetlin patterns $A$ with top row $\lambda$; and the number of
$d$-dimensional efficient cells of the decomposition $\cC$ is $\sum_{A}\binom{|\Free(A)|}{d}$.
\end{corollary}

\begin{proof}
Immediate from Theorem~\ref{thm:surj} and Definition~\ref{def:free}, since the marked patterns over
a fixed $A$ are in bijection with the subsets of $\Free(A)$, and those of rank $d$ with the subsets
of size $d$.
\end{proof}

\section{The example \texorpdfstring{$\lambda=(2,1,0)$}{lambda = (2,1,0)}}
\label{sec:example}

We work out the bijection completely for $n=3$ and $\lambda=(2,1,0)$. Everything in this section is
a specialization of Section~\ref{sec:bijection} and Section~\ref{sec:surjectivity}; the example is
small enough to be listed in full, and the cells that appear in it will be used again in
Section~\ref{sec:crystal}.

Here $d=3$ and $\GZ(2,1,0)\subset\RR^3$ has coordinates $(y_{11},y_{12},y_{21})$ arranged as
\[
\begin{matrix}
0 && 1 && 2\\
 & y_{11} && y_{12} & \\
 && y_{21} &&
\end{matrix}
\qquad\qquad
0\le y_{11}\le 1,\quad 1\le y_{12}\le 2,\quad y_{11}\le y_{21}\le y_{12}.
\]
Under the dictionary \eqref{eq:dict} we have $y_{11}=\nu^{(2)}_2$, $y_{12}=\nu^{(2)}_1$ and
$y_{21}=\nu^{(1)}_1$, with $\nu^{(3)}=\lambda$.

\subsection{Patterns and free positions}
\label{ss:ex-patterns}

By Definition~\ref{def:free} the three positions are free subject to
\[
y_{11}>0,\qquad y_{12}>1,\qquad y_{21}>y_{11},
\]
the upper-left neighbors of the three entries being $0$, $1$ and $y_{11}$ respectively. The eight
Gelfand--Zetlin patterns and their free positions are as follows.

\begin{center}
\begin{tabular}{cccc}
\toprule
$(a_{11},a_{12},a_{21})$ & free positions & $|\Free|$ & \# marked patterns\\
\midrule
$(0,1,0)$ & --- & $0$ & $1$\\
$(0,1,1)$ & $y_{21}$ & $1$ & $2$\\
$(0,2,0)$ & $y_{12}$ & $1$ & $2$\\
$(0,2,1)$ & $y_{12},y_{21}$ & $2$ & $4$\\
$(0,2,2)$ & $y_{12},y_{21}$ & $2$ & $4$\\
$(1,1,1)$ & $y_{11}$ & $1$ & $2$\\
$(1,2,1)$ & $y_{11},y_{12}$ & $2$ & $4$\\
$(1,2,2)$ & $y_{11},y_{12},y_{21}$ & $3$ & $8$\\
\midrule
\multicolumn{3}{r}{total} & $27$\\
\bottomrule
\end{tabular}
\end{center}

By Corollary~\ref{cor:count} the numbers of efficient cells of $\cC$ in each dimension are
\[
(f_0,f_1,f_2,f_3)=(8,\,12,\,6,\,1),
\qquad
f_0-f_1+f_2-f_3=8-12+6-1=1,
\]
as it must be for a subdivision of a $3$-ball. On the other side of the bijection one counts
$|\SVT_3(2,1)|$ by the value of $\max T(1,1)$: that maximum is $1$ in $1\cdot 7\cdot 3=21$ tableaux
and $2$ in $2\cdot3\cdot1=6$, and it cannot be $3$; total $27$, in agreement.

\begin{remark}\label{rem:notface}
The decomposition $\cC$ is a genuine subdivision and not the face decomposition of the polytope.
Indeed $\GZ(2,1,0)$ has eight lattice points, which by Theorem~\ref{thm:ps} are the eight
zero-dimensional cells, but only seven vertices: at $(y_{11},y_{12},y_{21})=(0,2,1)$ exactly two of
the six defining inequalities are active, namely $y_{11}\ge0$ and $y_{12}\le2$, so this lattice
point lies in the relative interior of a two-dimensional face.
\end{remark}

\subsection{The complete correspondence}
\label{ss:ex-table}

We write a set-valued tableau of shape $(2,1)$ as $A\,B/C$ with $A=T(1,1)$, $B=T(1,2)$,
$C=T(2,1)$, listing the elements of each set without braces; thus $12\,23/3$ denotes
$A=\{1,2\}$, $B=\{2,3\}$, $C=\{3\}$. The following table lists all $27$ marked patterns together
with the corresponding tableaux $\Phi^{-1}$.

\begin{center}
\small
\begin{tabular}{cccl@{\qquad\qquad}cccl}
\toprule
pattern & non-encircled & $\rk$ & $T$ & pattern & non-encircled & $\rk$ & $T$\\
\midrule
$(0,1,0)$ & --- & $0$ & $2\,3/3$      & $(1,1,1)$ & --- & $0$ & $1\,3/2$\\
$(0,1,1)$ & --- & $0$ & $1\,3/3$      & $(1,1,1)$ & $y_{11}$ & $1$ & $1\,3/23$\\
$(0,1,1)$ & $y_{21}$ & $1$ & $12\,3/3$ & $(1,2,1)$ & --- & $0$ & $1\,2/2$\\
$(0,2,0)$ & --- & $0$ & $2\,2/3$      & $(1,2,1)$ & $y_{11}$ & $1$ & $1\,2/23$\\
$(0,2,0)$ & $y_{12}$ & $1$ & $2\,23/3$ & $(1,2,1)$ & $y_{12}$ & $1$ & $1\,23/2$\\
$(0,2,1)$ & --- & $0$ & $1\,2/3$      & $(1,2,1)$ & $y_{11},y_{12}$ & $2$ & $1\,23/23$\\
$(0,2,1)$ & $y_{12}$ & $1$ & $1\,23/3$ & $(1,2,2)$ & --- & $0$ & $1\,1/2$\\
$(0,2,1)$ & $y_{21}$ & $1$ & $12\,2/3$ & $(1,2,2)$ & $y_{11}$ & $1$ & $1\,1/23$\\
$(0,2,1)$ & $y_{12},y_{21}$ & $2$ & $12\,23/3$ & $(1,2,2)$ & $y_{12}$ & $1$ & $1\,13/2$\\
$(0,2,2)$ & --- & $0$ & $1\,1/3$      & $(1,2,2)$ & $y_{21}$ & $1$ & $1\,12/2$\\
$(0,2,2)$ & $y_{12}$ & $1$ & $1\,13/3$ & $(1,2,2)$ & $y_{11},y_{12}$ & $2$ & $1\,13/23$\\
$(0,2,2)$ & $y_{21}$ & $1$ & $1\,12/3$ & $(1,2,2)$ & $y_{11},y_{21}$ & $2$ & $1\,12/23$\\
$(0,2,2)$ & $y_{12},y_{21}$ & $2$ & $1\,123/3$ & $(1,2,2)$ & $y_{12},y_{21}$ & $2$ & $1\,123/2$\\
 & & & & $(1,2,2)$ & $y_{11},y_{12},y_{21}$ & $3$ & $1\,123/23$\\
\bottomrule
\end{tabular}
\end{center}

\begin{example}\label{ex:full}
Take the last line. For $T=1\,123/23$ the minimum tableau is $1\,1/2$, so $\nu^{(1)}=(2,0)$ and
$\nu^{(2)}=(2,1)$, giving $(a_{11},a_{12},a_{21})=(1,2,2)$. The non-minimal elements are $2$ and $3$
in the cell $(1,2)$ and $3$ in the cell $(2,1)$; the excessive pairs are $(1,1)$, $(2,1)$ and
$(2,2)$, corresponding to all three positions, so no entry is encircled and $\rk=3$. By
Construction~\ref{constr:cells} the associated cell is the interior of the polytope. Conversely,
$\ex(T)=3$ and $\wt(T)=(2,2,2)$, in agreement with $x^P=\beta^3x_1^2x_2^2x_3^2$.
\end{example}

\begin{example}\label{ex:cells}
Two cells that will be needed in Section~\ref{ss:counterexample}. For $T=1\,23/23$ the pattern is
$(1,2,1)$ with $y_{11}$ and $y_{12}$ not encircled, so the entry $a_{21}=1$ is encircled and, equal
to its upper-left neighbor $a_{11}=1$ but not to its upper-right neighbor $a_{12}=2$, is joined to
the left. Construction~\ref{constr:cells} imposes $y_{21}=y_{11}$, together with $0<y_{11}<1$ and
$1<y_{12}<2$, so
\[
C=\bigl\{(u,t,u)\ :\ 0<u<1,\ 1<t<2\bigr\},
\qquad \dim C=2=\ex(T).
\]
For $T=1\,23/3$ the pattern is $(0,2,1)$ with only $y_{12}$ not encircled. Here $a_{11}=0$ is joined
to its upper-left neighbor $a_{0,1}=0$, giving $y_{11}=0$, while $a_{21}=1$ equals neither
$a_{11}=0$ nor $a_{12}=2$ and so carries no edge, giving $y_{21}=1$. Hence
\[
C=\bigl\{(0,t,1)\ :\ 1<t<2\bigr\},
\qquad \dim C=1=\ex(T).
\]
\end{example}


\section{The square-root crystal and the face poset}
\label{sec:crystal}

The bijection $\widetilde\Phi$ transports to $\cP^+(\lambda)$, and hence to the cells of $\cC$, the
crystal structure that Yu puts on $\SVT_n(\lambda)$. It is natural to ask whether the transported
structure is compatible with the geometry of $\GZ(\lambda)$ --- an expectation voiced at the end of
the introduction of~\cite{PresnovaSmirnov24}. In this section we prove that one such compatibility
does hold (Theorem~\ref{thm:altern}) and that the most natural further one does not
(Theorem~\ref{thm:counter}).

We first recall the operators. Fix $i\in[n-1]$ and $T\in\SVT_n(\lambda)$, and read the cells of
$\lambda$ in \emph{column order}: from the leftmost column to the rightmost, and within each column
from the bottom row to the top. Each cell contributes to a word: a symbol \verb")" if it contains $i$ but
not $i+1$, a symbol \verb"(" if it contains $i+1$ but not $i$, the string \verb")-(" if it contains both,
and nothing otherwise. Matching brackets in the usual way, ignoring the dashes, partitions this word
into \emph{classes}; a class is of \emph{right form} if it begins with an unmatched \verb")" and does not
end with an unmatched \verb"(", of \emph{left form} in the mirror situation, and of \emph{combined form}
if both. The operator $f'_i$ acts as follows: if a class of combined form occurs, delete $i$ from
the cell contributing the initial \verb")" of the first such class; otherwise add $i+1$ to the cell
contributing the final \verb")" of the last class of right form, and set $f'_i(T)=0$ if there is none.
The operator $e'_i$ is defined symmetrically. We refer to~\cite[\S4]{Yu21} for the precise
formulation and for the proof that these are well defined; the point for us is that $(f'_i)^2=f_i$
and $(e'_i)^2=e_i$ are the usual crystal operators. A \emph{double $i$-string} is a maximal sequence
$T_0,\dots,T_{2k}$ with $e'_i(T_0)=f'_i(T_{2k})=0$ and $f'_i(T_j)=T_{j+1}$; the set
$\SVT_n(\lambda)$ is the disjoint union of its double $i$-strings.

Throughout this section we abbreviate $C_T:=C_{\widetilde\Phi(T)}$.

\subsection{Dimension alternation along a double \texorpdfstring{$i$}{i}-string}
\label{ss:alternation}

\begin{theorem}\label{thm:altern}
Let $T_0,\dots,T_{2k}$ be a double $i$-string. Then
\[
\dim C_{2j+1}=\dim C_{2j}+1=\dim C_{2j+2}+1
\qquad (0\le j\le k-1),
\]
where $C_j:=C_{T_j}$; equivalently $\dim C_{T_j}=\ex(T_0)+(j\bmod2)$ for all $j$. In particular a
double $i$-string has odd length, and its cells alternate between two consecutive dimensions.
\end{theorem}

\begin{proof}
By~\cite[Lemma~4.22]{Yu21} the weights along a double $i$-string satisfy
$\wt(T_{2j+1})=\wt(T_{2j})+v_{i+1}$ and $\wt(T_{2j+2})=\wt(T_{2j+1})-v_i$, where $v_m$ is the $m$-th
standard basis vector; that is, $f'_i$ alternately adjoins the letter $i+1$ and deletes the letter
$i$. Hence $\ex(T_{2j+1})=\ex(T_{2j})+1$ and $\ex(T_{2j+2})=\ex(T_{2j+1})-1$. By
Theorem~\ref{thm:surj} we have $\dim C_{T}=\ex(T)$ for every $T$.
\end{proof}

\begin{corollary}\label{cor:altern}
The operator $f'_i$ never maps a cell to a cell of the same dimension, whereas the ordinary crystal
operators $f_i=(f'_i)^2$ and $e_i=(e'_i)^2$ preserve the dimension of the cell.
\end{corollary}

\subsection{Consecutive cells need not be incident}
\label{ss:counterexample}

Theorem~\ref{thm:altern} makes the following question natural, and the first examples encourage it.
For instance, in the double $2$-string
\[
1\,2/2\ \to\ 1\,23/2\ \to\ 1\,3/2\ \to\ 1\,3/23\ \to\ 1\,3/3
\]
for $\lambda=(2,1,0)$ the five cells are, in order, the point $(1,2,1)$; the open segment
$y_{11}=y_{21}=1$, $1<y_{12}<2$; the point $(1,1,1)$; the open segment $0<y_{11}<1$,
$y_{12}=y_{21}=1$; and the point $(0,1,1)$. Each segment has the two neighboring points as its
endpoints, so this double $2$-string is a zigzag path in the $1$-skeleton of $\cC$.

\begin{question}\label{q:facet}
Is $C_{2j}$, respectively $C_{2j+2}$, always a face of $C_{2j+1}$, so that a double $i$-string is a
zigzag path in the Hasse diagram of the face poset of $\cC$?
\end{question}

\begin{theorem}\label{thm:counter}
The answer is negative. For $\lambda=(2,1,0)$, $n=3$ and $i=2$, consider the double $2$-string
\[
T_0=1\,2/23\ \xrightarrow{\ f'_2\ }\ T_1=1\,23/23\ \xrightarrow{\ f'_2\ }\ T_2=1\,23/3 .
\]
Then $\overline{C_{T_1}}\cap\overline{C_{T_2}}=\varnothing$; in particular $C_{T_2}$ is not a face of
$C_{T_1}$.
\end{theorem}

\begin{proof}
We first check the string. In column order the cells are $T(2,1),T(1,1),T(1,2)$. For
$T_0=1\,2/23$ the sets are $\{1\},\{2\},\{2,3\}$, so the $2$-word is \verb")-()": a single class,
which begins with an unmatched \verb")" and ends with an unmatched \verb")", hence of right form and not
combined. Therefore $f'_2$ adjoins the letter $3$ to the cell contributing the final \verb")", namely
$T_0(1,2)=\{2\}$, giving $T_1=1\,23/23$. The $2$-word of $T_1$ is \verb")-()-(", a single class of
combined form, so $f'_2$ deletes the letter $2$ from the cell contributing the initial \verb")", namely
$T_1(2,1)=\{2,3\}$, giving $T_2=1\,23/3$. Finally the $2$-word of $T_2$ is \verb"()-(", of left form,
so $f'_2(T_2)=0$ and the string terminates.

The two cells were computed in Example~\ref{ex:cells}:
\[
C_{T_1}=\bigl\{(u,t,u):0<u<1,\ 1<t<2\bigr\},
\qquad
C_{T_2}=\bigl\{(0,t,1):1<t<2\bigr\},
\]
of dimensions $2$ and $1$ respectively, in accordance with $\ex(T_1)=2$ and $\ex(T_2)=1$. Every
point of $\overline{C_{T_1}}$ satisfies $y_{21}=y_{11}$, whereas every point of
$\overline{C_{T_2}}$ satisfies $y_{11}=0$ and $y_{21}=1$. The two closures are therefore disjoint:
the cells are not merely non-incident, they are not even adjacent.
\end{proof}

\begin{remark}
The failure is not an artefact of the tie-breaking rule in the all-equal branch of $\cR$: neither of
the two patterns occurring in the proof contains an all-equal triangle, so no tie-break is involved.
\end{remark}

\subsection{Open questions}
\label{ss:questions}

We collect the questions left open, in decreasing order of how close they seem to the results
proved above.

The first concerns the operators themselves. Yu's $f'_i$ is defined by reading a word off the
tableau, and the description we transported to $\cP^+(\lambda)$ inherits that detour: to apply
$f'_i$ to a pattern one passes to $\widetilde\Phi^{-1}$ of it, acts, and comes back.

\begin{question}\label{q:intrinsic}
Is there an intrinsic description of $f'_i$, for general $i$, in terms of the enhanced pattern or of
the cell $C_P$ --- one that does not pass through set-valued tableaux? Equivalently, is there a
description of the square-root crystal structure that is visibly compatible with the subdivision
$\cC$?
\end{question}

The obstruction is what Section~\ref{ss:counterexample} exhibits: for $i\ge2$ the operator can move
between cells that are not even adjacent, so no local rule on the face poset can produce it. This
suggests weakening what one asks of the geometry.

\begin{question}\label{q:better}
Is there a cellular decomposition of $\GZ(\lambda)$, still indexing the monomials of
$G^{(\beta)}_\lambda$ with the multiplicities of \eqref{eq:ps}, on which the operators $f'_i$ act by
passing to a face or a coface? A candidate would have to differ from $\cC$: by
Theorem~\ref{thm:counter} the decomposition of~\cite{PresnovaSmirnov24} does not have this property.
\end{question}

The remaining questions concern the bijection rather than the crystal. The most substantial is
whether $\widetilde\Phi$ sees Lascoux polynomials and not merely Grassmannian Grothendieck
polynomials. Both of the rules being compared here have such a refinement. On the tableau side, Yu's
theorem expresses a Lascoux polynomial as a sum over those set-valued tableaux of shape $\lambda$
whose right key is bounded by a fixed key~\cite{Yu21}. On the polytope side,
\cite[Theorem~4.3.3]{PresnovaSmirnov24} expresses it as
\[
\cL_{w,\lambda}=\sum_{P\in\cP^+(w,\lambda)}x^P,
\]
where $\cP^+(w,\lambda)$ consists of the efficient patterns whose cells lie in the union of the dual
Kogan faces of $\GZ(\lambda)$ attached to the permutation $w$. Since both sums compute the same
polynomial and $\widetilde\Phi$ matches monomials term by term by Theorem~\ref{thm:surj}, the two
subsets carry the same multiset of monomials; but that is a numerical coincidence and does not by
itself relate the subsets.

\begin{question}\label{q:lascoux}
Does $\widetilde\Phi$ carry the set of tableaux selected by the right-key condition for $w$ onto
$\cP^+(w,\lambda)$? Equivalently, is the right key of $T$ computable from the position of the cell
$C_{\widetilde\Phi(T)}$ relative to the dual Kogan faces of $\GZ(\lambda)$?
\end{question}

We strongly believe that the answer is affirmative; the details will appear
elsewhere~\cite{SmirnovKeys}. This will
upgrade Corollary~\ref{cor:equiv} from an equivalence between two formulas for $G^{(\beta)}_\lambda$
to an equivalence between two rules for all Lascoux polynomials, and will identify the right key
--- a construction of a combinatorial nature --- with a geometric position inside the polytope.

Finally, two smaller matters. The bijection is defined on $\cP^+(\lambda)$, and inefficient patterns
are simply discarded; under Corollary~\ref{cor:whyeff} they correspond to data that would-be
tableaux cannot carry, the prescribed excess letter colliding with the cell below its host.

\begin{question}\label{q:inefficient}
Is there a combinatorial model for the inefficient patterns themselves, compatible with
$\widetilde\Phi$ --- for instance one in which the cells of $\cC$ that carry no monomial are
indexed by fillings violating column-strictness in a controlled way? What is the topology of the
union of the inefficient cells?
\end{question}

\begin{question}\label{q:othertypes}
Gelfand--Zetlin polytopes have analogues for the other classical groups, and Grothendieck
polynomials have $K$-theoretic analogues in those settings. Is there a counterpart of
Theorem~\ref{thm:surj} there --- a subdivision of the relevant string or
Berenstein--Zelevinsky polytope whose cells are indexed by an appropriate notion of set-valued
tableau?
\end{question}

\providecommand{\bysame}{\leavevmode\hbox to3em{\hrulefill}\thinspace}
\providecommand{\MR}{\relax\ifhmode\unskip\space\fi MR }
\providecommand{\MRhref}[2]{%
  \href{http://www.ams.org/mathscinet-getitem?mr=#1}{#2}
}
\providecommand{\href}[2]{#2}

\end{document}